\documentclass[10pt,a4paper]{article}
\usepackage[final]{optional}
\usepackage[british,american]{babel}

\usepackage{amsfonts, amsmath, wasysym}
\usepackage{amssymb,amsthm,
paralist
}
\usepackage{
latexsym,
nicefrac,
}
\usepackage[usenames]{color}

\usepackage{url}

\definecolor{darkgreen}{rgb}{0,0.5,0}
\definecolor{darkred}{rgb}{0.7,0,0}
\usepackage[colorlinks, 
citecolor=darkgreen, linkcolor=darkred
]{hyperref}
\usepackage{esint}
\usepackage{bibgerm}
\usepackage[normalem]{ulem}

\theoremstyle{plain}
\newtheorem{lemma}{Lemma}[section]
\newtheorem{thm}[lemma]{Theorem}
\newtheorem{prop}[lemma]{Proposition}

\theoremstyle{definition}
\newtheorem{defn}[lemma]{Definition}

\newtheorem{rmk}[lemma]{Remark}
\def\blbox{\quad \vrule height7.5pt width4.17pt depth0pt}
\newcommand{\cmt}[1]{\opt{draft}{\textcolor[rgb]{0.5,0,0}{
$\LHD$ #1 $\RHD$\marginpar{\blbox}}}}

\numberwithin{equation}{section}

\newcommand{\m}{\ensuremath{{\cal M}}}

\newcommand{\pl}[2]{{\frac{\partial #1}{\partial #2}}}
\newcommand{\al}{\alpha}
\newcommand{\be}{\beta}
\newcommand{\ga}{\gamma}

\newcommand{\Om}{\Omega}

\newcommand{\si}{\sigma}

\newcommand{\Si}{\Sigma}
\newcommand{\Th}{\Theta}

\newcommand{\vph}{\varphi}
\newcommand{\ep}{\varepsilon}

\newcommand{\R}{\ensuremath{{\mathbb R}}}
\newcommand{\N}{\ensuremath{{\mathbb N}}}

\newcommand{\C}{\ensuremath{{\mathbb C}}}

\newcommand{\Hyp}{\ensuremath{{\mathbb H}}}

\newcommand{\weakto}{\rightharpoonup}

\newcommand{\embed}{\hookrightarrow}

\newcommand{\tensor}{\otimes}

\newcommand{\grad}{\nabla}

\newcommand{\norm}[1]{\left\Vert#1\right\Vert}  

\def\blbox{\quad \vrule height7.5pt width4.17pt depth0pt}

\newcommand{\beq}{\begin{equation}}
\newcommand{\eeq}{\end{equation}}
\newcommand{\beqa}{\begin{equation}\begin{aligned}}
\newcommand{\eeqa}{\end{aligned}\end{equation}}
\newcommand{\brmk}{\begin{rmk}}
\newcommand{\ermk}{\end{rmk}}
\newcommand{\partref}[1]{\hbox{(\csname @roman\endcsname{\ref{#1}})}}
\newcommand{\half}{\frac{1}{2}}

\newcommand{\lie}{{\cal L}}

\newcommand{\avint}{{\int\!\!\!\!\!\!-}}

\newcommand{\dbar}{{\overline\partial}}
\newcommand{\zbar}{{\overline z}}
\newcommand{\PSL}{{\mathrm{PSL}}}

\newcommand*\tr{\mathop{\mathrm{tr}}\nolimits}

\newcommand{\M}{\ensuremath{{\mathcal M}}_{-1}}
\newcommand{\A}{\ensuremath{{\mathcal A}}}
\newcommand{\D}{\ensuremath{{\mathcal D}}}

\newcommand{\abs}[1]{\left\vert#1\right\vert} 
 
\newcommand{\eps}{\varepsilon}
\newcommand{\na}{\nabla}

\title{{\sc
Flowing maps to minimal surfaces
}
\\ 
}
\author{Melanie Rupflin and Peter M. Topping}
\date{\today}
\begin{document}
\maketitle

\begin{abstract}
We introduce a flow of maps from a compact surface of arbitrary genus to an arbitrary Riemannian manifold which has elements in common with both the harmonic map flow and the mean curvature flow, but is more effective at finding minimal surfaces.
In the genus 0 case, our flow is just the harmonic map flow, and it tries to find branched minimal 2-spheres
as in Sacks-Uhlenbeck \cite{SU} and Struwe \cite{struweCMH} etc.
In the genus 1 case, we show that our flow is exactly equivalent to that considered by Ding-Li-Liu \cite{DLL}.
In general, we recover the result of Schoen-Yau \cite{SY} and Sacks-Uhlenbeck \cite{SU_TAMS} that an incompressible map from a surface can be adjusted to a branched minimal immersion with the same action on $\pi_1$, and this minimal immersion will be homotopic to the original map in the case that $\pi_2=0$.
\end{abstract}

\section{Introduction}
\label{intro}

\subsection{Overview}

Given a map from a compact surface to a compact Riemannian manifold, one can consider the possibility of adjusting it somehow in order to find a minimal immersion. 
A first approach might be to start with an embedding and deform it to reduce its area, and taking this route one is led to the mean curvature flow. One can view this flow as the harmonic map flow where the domain metric is adjusted continually so that the map remains an isometry. 
Unfortunately, this flow develops complicated singularities, without additional hypotheses.

An alternative is to try to find critical points of the harmonic map energy with respect to variations of both the map \emph{and} the domain metric, which leads directly to 
a weakly conformal harmonic map, which (if nonconstant) is then a branched minimal immersion \cite{GOR}. A variational approach 
along these lines, minimising with respect to maps and then conformal structures, was successfully taken by Schoen-Yau \cite{SY} and Sacks-Uhlenbeck \cite{SU_TAMS}; the strategy was implicitly used in the solution of Douglas-Rado to the Plateau problem and its later variants
\cite{GM}.

In this paper, we show how a flow can find this branched minimal immersion in one go by simultaneously flowing both the map and the
domain in a coupled way,
following the gradient flow of the harmonic map energy. 
Intuitively, one would mainly like to evolve the domain by deforming its conformal structure, and we achieve that by flowing a Riemannian metric on the domain surface within the class of constant curvature metrics (with fixed area in the genus 1 case).
When the domain is a sphere, there is only one conformal structure available (modulo diffeomorphisms) and our flow will be seen to correspond to the classical harmonic map flow 
of Eells-Sampson \cite{ES} which evolves a map by the $L^2$-gradient flow of the harmonic map energy. 
In the special case that the domain is a torus, we will show that the flow agrees exactly with a flow introduced by Ding-Li-Lui \cite{DLL} which those authors had previously presented 
as a technical modification of an alternative flow.
In the higher genus case, the flow is new, and is a little harder to deal with because of the more complicated structure of Teichm\"uller space. Nevertheless, when combined with an existence theorem which will appear elsewhere \cite{rupflin_existence}, we show that incompressible maps flow for all time without the domain metric degenerating, and subconverge to a conformal harmonic map.
\subsection{Definition of the flow}
\label{flow_def_sect}

Let $M:=M_\gamma$ be a smooth closed orientable surface of genus $\gamma\geq 0$ equipped with a smooth metric $g$. 
Let $N=(N,G)$ be a smooth closed Riemannian manifold of any dimension, which we may view as being isometrically immersed in $\R^K$ for some $K\in \N$.

The harmonic map energy of a smooth map 
$u:(M,g)\to (N,G)$ is defined to be
$$E(u,g):=\half\int_M |du|^2d\mu_g.$$
Since the energy is invariant under conformal variations of the domain metric, we will restrict our attention to the set of hyperbolic (constant Gauss curvature $-1$) metrics $\M$ if $\gamma\geq 2$,
or to the flat metrics $\m_0$ of unit total area if $\gamma=1$, or to the metrics $\m_1$ of curvature $+1$ if $\gamma=0$. To consider general genus, we will sometimes unify the notation by writing $\m_c$ for $\M$, $\m_0$ or $\m_1$ as appropriate.

Because the energy remains invariant if we pull back both the map and the domain metric by the same diffeomorphism $f$, i.e. $E(u,g)=E(u\circ f,f^* g)$,
it is natural to identify pairs $(u_1,g_1)$ and $(u_2,g_2)$ if there exists a smooth diffeomorphism $f$ of $M$ 
homotopic to the identity  
such that $u_1=u_2\circ f$ and $g_1=f^* g_2$ and to study the energy
on 
$$\A=\{[(u,g)]: \, g\in \m_c, u\in C^\infty(M,N)\}.$$

Although $\A$ is defined here as a \emph{set}, if we follow the path laid out in Teichm\"uller theory (see \cite[Chapter 2]{tromba}) %page 50
we can see a slight variant of $\A$ as a smooth infinite-dimensional manifold, and endow it with a natural metric structure. 
We will clarify some subtleties of that process in Appendix \ref{banach_manifold}, but in order to derive the new flow equations it suffices to consider formally a metric structure on $\A$ by taking an arbitrary point $[(u_0,g_0)]\in\A$, considering a path 
$(u(t),g(t))\in C^{\infty}(M,N)\times \m_c$ with 
$(u(0),g(0))=(u_0,g_0)$ and imagining how we should define the length of the `tangent vector' $(\pl{u}{t}(0),\pl{g}{t}(0))$.

The part $\pl{g}{t}(0)$ can be viewed as a 
tangent vector at $g_0$ in $\m_c$. 
This can be written (see for example \cite{tromba} and Appendix
\ref{banach_manifold}) as
$Re (\theta dz^2)+\lie_X g$ for some holomorphic quadratic differential $\Theta=\theta dz^2$ and vector field $X$ (where $z$ is a complex coordinate induced by $g_0$, and we adopt the convention that $dz^2=dz\tensor dz$) as described in \cite{tromba}, say.
The splitting of tangent vectors into `horizontal' part $Re (\theta dz^2)$ and `vertical' part $\lie_X g$ is a canonical $L^2$-orthogonal decomposition and does not depend on the choice of complex coordinate
(although $X$ itself will not be unique in the genus 0 or 1 case since we can add any Killing field). 
If we then modify the path $(u(t),g(t))$ by an appropriate diffeomorphism, we will be able to assume that $\pl{g}{t}(0)$ is entirely horizontal. More precisely, if we let $f_t$ be the family of diffeomorphisms $M\mapsto M$ generated by $-X$, with
$f_0$ the identity (for $t$ in a 
neighbourhood of $0$) so that 
$-\lie_X g_0=\frac{d}{dt}|_{t=0}f_t^*g_0$
then we may define 
$\tilde u(t)=u(t)\circ f_t$ and $\tilde g(t)=f_t^*g(t)$ 
and find that
$\pl{\tilde g}{t}(0)=-\lie_X g_0+\pl{g}{t}(0)
=Re (\theta dz^2)$
(i.e. it is horizontal) and $(\tilde u(t),\tilde g(t))$
generates the same path through $\A$ as $(u(t),g(t))$.

Now we can define an inner product on $\cal A$: for a tangent vector $V$ at $[(u,g)]$ represented by 
$(\pl{u}{t},Re (\theta dz^2))$, the inner product is determined by the norm defined by
$$\|V\|^2:=\bigg\|\pl{u}{t}\bigg\|_{L^2}^2+\eta^{-2}\|Re(\theta dz^2)\|_{L^2}^2,$$
for some constant $\eta>0$.
The metric $g$ also induces a Hermitian inner product on quadratic differentials, and this allows us to write
\begin{equation}
\label{real_complex}
\|\theta dz^2\|_{L^2}^2=2\|Re(\theta dz^2)\|_{L^2}^2.
\end{equation}
(Note that if 
$g=
\rho^{2}(dx^2+ dy^2)$ then $|dz^2|=2\rho^{-2}$.)

Now that we are equipped with a metric structure on $\A$, albeit at a formal level, we are in a position to consider the gradient flow of $E$ on $\A$, and to this end we
consider the differential of the energy $E$ viewed as a functional on $\cal A$.
Then we have the formula
$$dE_{(u,g)}\left(\pl{u}{t},\pl{g}{t}\right)=-\int_M \langle \tau_g(u),\pl{u}{t}\rangle+\frac{1}{4}\langle Re(\Phi(u,g)),\pl{g}{t}\rangle$$
where $\Phi(u,g)=\phi dz^2:=4(u^*G)^{(2,0)}$ is the Hopf differential of $u$,
that is, $\phi=|u_x|^2-|u_y|^2-2i\langle u_x,u_y\rangle$,
and $\tau_g(u):=\tr\grad du$ is the tension field \cite{ES}.
Writing $P_g$ for the $L^2$-orthogonal projection of the space of quadratic differentials onto the space of holomorphic quadratic differentials (with respect to the Hermitian inner product induced by $g$) and restricting $\pl{g}{t}$ to be horizontal (i.e. to be the real part of a \emph{holomorphic} $\theta dz^2$) we have

\begin{equation}
\begin{aligned}
dE_{(u,g)}\left(\pl{u}{t},Re (\theta dz^2)\right)&=-\int_M \langle \tau_g(u),\pl{u}{t}\rangle+\frac{1}{4}\langle Re(\phi dz^2),Re(\theta dz^2)\rangle\\
&=-\int_M \langle \tau_g(u),\pl{u}{t}\rangle+\frac{1}{4}\langle Re(P_g(\phi dz^2)),Re(\theta dz^2)\rangle\\
&=-\big\langle (\tau_g(u),\frac{\eta^2}{4} Re(P_g(\phi dz^2))),(\pl{u}{t},Re(\theta dz^2))\big\rangle_{\cal A}.
\end{aligned}
\end{equation}
Here we have used that if two quadratic differentials are orthogonal with respect to the $L^2$ Hermitian inner product, then their real parts are also orthogonal. 

We can then write the gradient flow of $E$ with respect to the inner product we have defined on $\cal A$ as
$[(u(t),g(t))]$, where
\begin{equation}
\label{flow}
\pl{u}{t}=\tau_g(u);\qquad \pl{g}{t}=\frac{\eta^2}{4} Re(P_g(\Phi(u,g))),
\end{equation}
and the energy decays according to
\begin{equation}
\label{energy_decay}
\frac{dE}{dt}=-\int_M |\tau_g(u)|^2+\left(\frac{\eta}{4}\right)^2 |Re(P_g(\Phi(u,g)))|^2.
\end{equation}
It would be reasonable to set $\eta=2$ at this stage, 
but we retain this generality for the moment because it will make the comparison with \cite{DLL} a little more transparent.

By Liouville's theorem, the space of holomorphic 
quadratic differentials on the sphere consists only of the zero element, and on the torus, its real dimension is 2 (consisting in this case only of parallel quadratic differentials).
By the Riemann-Roch theorem, the space of holomorphic quadratic differentials for genus $\gamma\geq 2$ 
is a vector space of real dimension $6\gamma-6$.

In particular, on the sphere, $g$ will be static in time, and the flow is precisely the harmonic map flow of Eells-Sampson \cite{ES}.
A global existence theorem for weak solutions of that flow from surfaces was given by Struwe \cite{struweCMH}.

On the torus, as we describe in Section \ref{torus_sect}, the flow \eqref{flow} is easier to handle than for higher genus because the flow moves the domain metric $g$ within a two-dimensional smooth submanifold of the infinite-dimensional space $\m_0$ of all unit-area flat metrics. The foundational theory for this case, including the existence theory, was given by Ding-Li-Liu \cite{DLL} from a somewhat different viewpoint.

For genus $\gamma\geq 2$, an extra complication arises. The domain metric $g$ moves now within the space $\m_{-1}$ of hyperbolic metrics, and although its velocity $\pl{g}{t}$ always lies within a $6\gamma-6$ dimensional subspace of the infinite dimensional tangent space $T\m_{-1}$, the distribution defined by these `horizontal' subspaces is no longer integrable. What this means in practice is that the domain metric $g(t)$ at two different times may represent the same point in Teichm\"uller space, but differ by pull-back under some diffeomorphism.

Nevertheless, by drawing on ideas from Teichm\"uller theory, we overcome this difficulty. In \cite{rupflin_existence}, the first author will prove the following existence and uniqueness theorem, effectively extending the theory of Struwe \cite{struweCMH} and Ding-Li-Liu \cite{DLL}:

\begin{thm}\label{thm:1}
\label{existence}
For any initial data $(u_0,g_0)\in  C^\infty(M,N)\times\M$ there exists a (weak) solution $(u,g)$ of \eqref{flow} defined on a maximal interval $[0,T)$, $T\leq \infty$, satisfying the following properties
\begin{enumerate}
\item[(i)] The solution $(u,g)$ is smooth away from at most finitely many singular times $T_i\in(0,T)$. Furthermore as $t\to T_i$ the metrics $g(t)$ converge smoothly to an element $g(T_i)\in \M$ and the maps $u(t)$ converge to $u(T_i)$ weakly in $H^1$ and smoothly away from a finite set of points $S(T_i)$ in $M$. 
\item[(ii)] The energy $t\mapsto E(u(t),g(t))$ is non-increasing. 
\item[(iii)] The solution exists until the metrics degenerate in moduli space; i.e. if the maximal existence time $T<\infty$ then the length $\ell(g(t))$ of the shortest closed geodesic of $(M,g(t))$ converges to zero as 
$t\nearrow T$.
\end{enumerate} 
Furthermore, this solution is uniquely determined by its initial data in the class of all weak solutions with non-increasing energy.
\end{thm}
\begin{defn}\label{def:weak}
We call $(u,g)\in H^1_{loc}(M\times[0,T),N)\times C^0([0,T),\M)$ a weak solution of \eqref{flow} on $[0,T)$, 
$T\leq \infty$, provided $u$ solves the first equation of \eqref{flow} in the sense of distributions and $g$ is piecewise $C^1$ 
(viewed as a map from $[0,T)$ into the space of symmetric (0,2) tensors equipped with any $C^k$ metric, $k\in\N$) 
and satisfies the second equation of \eqref{flow} away from times where it is not differentiable.
\end{defn}

\begin{rmk}
As for the harmonic map flow, the singularities of the map component are due to the ``bubbling off of minimal spheres'' as $t\nearrow T_i$ and thus cause an instantaneous loss of energy, compare \cite{rupflin_existence}.
In the case that the genus $\ga<2$, there can be no degeneration
at finite time of the type described in (iii), as shown by Ding-Li-Liu \cite{DLL}, which can be viewed as a consequence of the completeness of 
Teichm\"uller space for tori with respect to the Weil-Petersson metric.
\end{rmk}

In the case that the metrics $g(t)$ do \emph{not} degenerate in moduli space, even at infinite time, we show in this paper that we get good asymptotic convergence, modulo parametrisation, to a branched minimal immersion or a constant map. 
More precisely we prove:

\begin{thm}
\label{asymptotics}
In the setting of Theorem \ref{existence}, if the length $\ell(g(t))$ of the shortest closed geodesic of $(M,g(t))$ is uniformly bounded below by a positive constant, then there exist a sequence of times $t_i\to\infty$, a sequence of orientation-preserving diffeomorphisms $f_i:M\to M$, a hyperbolic metric $\bar g$ on $M$, a weakly conformal harmonic map $\bar u:(M,\bar g)\to N$ and a finite set of points $S\subset M$ such that 
\begin{enumerate}
\item[(i)]
$f_i^*[g(t_i)]\to \bar g$ smoothly;
\item[(ii)]
$u(t_i)\circ f_i \weakto \bar u$ weakly in $H^1(M)$;
\item[(iii)]
$u(t_i)\circ f_i \to \bar u$ strongly in $W^{1,p}_{loc}(M\backslash S)$ for any $p\in [1,\infty)$, and thus also in $C^0_{loc}(M\backslash S)$;
\item[(iv)] the map $\bar u$ has the same action on $\pi_1(M)$ as $u_0$ -- see \eqref{def:action} below.
\end{enumerate}
\end{thm}

At each point in $S$ where the convergence to $\bar u$ fails to be smooth, there will be bubbling of one or more minimal spheres (cf. \cite{struwe_book}).

Also in this paper we demonstrate how this flow, having such asymptotics, can be used to prove the following minimal surface existence theorem of Schoen-Yau \cite{SY} and Sacks-Uhlenbeck \cite{SU_TAMS}. The genus 1 version was already addressed in \cite{DLL}.

\begin{thm}
\label{minsurfexist}
Let $M$ be a smooth closed orientable surface, $N=(N,G)$ a smooth closed Riemannian manifold and $u:M\to N$ an incompressible map.
Then there exists a branched minimal immersion $\bar u:M\to N$ (with respect to some conformal structure on $M$) with the same action on $\pi_1$ as $u$. Moreover, $\bar u$ can be chosen to be homotopic to $u$ in the case that $\pi_2(N)=0$.
\end{thm}

\begin{rmk}
To clarify, a map $u:M\to N$ is called \emph{incompressible} if it is continuous, homotopically nontrivial and its action on $\pi_1$ has trivial kernel. 
In fact, even if 
we only assume that the images of all \emph{simple} closed non-homotopically trivial curves are again non-homotopically trivial,
we can still find the branched minimal immersion $\bar u$ with the same action on $\pi_1$. 
Two maps $u^1:M\to N$ and $u^2:M\to N$ are said to have the same action on $\pi_1$ if for a fixed $x\in M$, the natural actions
$$u^1_*:\pi_1(M,x)\to\pi_1(N,u^1(x))\text{ and }
u^2_*:\pi_1(M,x)\to\pi_1(N,u^2(x))$$
are equivalent up to a change-of-basepoint isomorphism 
$\alpha_*:\pi_1(N,u^2(x))\to \pi_1(N,u^1(x))$ induced by some path 
$\alpha:[0,1]\to N$ from $u^1(x)$ to $u^2(x)$, i.e. if 
\beq \label{def:action} 
\alpha_*(u_*^1([\gamma])):=\alpha^{-1}\cdot u_*^1([\gamma])\cdot \alpha=u_* ^2([\gamma]) \eeq
for every $[\gamma]\in \pi_1(M,x)$. 
In that case we write $u^1_*\sim u^2_*$.
In particular, the images of an arbitrary closed curve in $M$ under $u^1$ and $u^2$ will be homotopic. 
For the purposes of this paper, a \emph{branched minimal immersion} is defined to be a nonconstant weakly conformal harmonic map. Such a map can be seen to be a smooth minimal immersion away from finitely many points in the domain where branching occurs \cite{GOR}.
\end{rmk}

\begin{rmk}
In \cite{SY} and \cite{SU_TAMS}, the authors obtain the existence of such a branched minimal surface which \emph{minimises} area over all maps with this action on $\pi_1$. In contrast, our method gives an essentially explicit deformation of the map (particularly when the target does not admit any minimal 2-spheres).
\end{rmk}

Some ideas from the proof of Theorems \ref{asymptotics} and \ref{minsurfexist} will be outlined now, and detailed in Section \ref{asymp_anal}.
The starting point is the flow from Theorem \ref{existence} which will be global under the hypotheses of Theorem \ref{minsurfexist}:
We will show that the conformal structure of the domain cannot degenerate in this case because if it did, the incompressibility of the map $u$ would force the energy to blow up rather than monotonically decrease.
It will then turn out that the flow 
deforms the map $u$ without altering its action on $\pi_1$, even as we pass over singular times.

By integrating the energy decay formula \eqref{energy_decay},
we see that 
\begin{equation}
\label{energy_decay_integrated}
\int_0^\infty\left(\int_M |\tau_g(u)|^2+\left(\frac{\eta}{4}\right)^2 |Re(P_g(\Phi(u,g)))|^2
\right)dt\leq E(u(0))-\lim_{t\to\infty}E(u(t))<\infty,
\end{equation}
and thus, keeping in mind \eqref{real_complex}, it is possible to pick times $t_i\to\infty$ such that 
$$\tau_{g(t_i)}(u(t_i))\to 0\text{ in }L^2\qquad\text{ and }\qquad P_{g(t_i)}(\Phi(u(t_i),g(t_i)))\to 0 \text{ in }L^2.$$
Standard estimates for the Hopf differential (see \eqref{hopfest1}) will tell us that $\dbar [\Phi(u(t_i),g(t_i))]\to 0$ in $L^1$. Putting this information into an elliptic-Poincar\'e estimate (Lemma \ref{Ptypelemma}) will then imply that $\Phi(u(t_i),g(t_i))\to 0$ in $L^1$.

A modified bubbling analysis together with Mumford's compactness theorem then tells us that after passing to a subsequence and pulling back by a sequence of orientation-preserving diffeomorphisms, the metrics $g(t_i)$ will converge to a limiting hyperbolic metric and the maps $u(t_i)$ converge to a finite energy conformal harmonic map $u_\infty$ (with respect to this limiting metric) away from a finite number of points in $M$.
After applying Sacks-Uhlenbeck's removable singularity theorem, we will see that $u_\infty$ is the branched minimal immersion that we seek.

This paper is organised as follows. In Section \ref{poincare} we state and prove an elliptic estimate for quadratic differentials which will be applied later to the Hopf differential. In Section \ref{asymp_anal} we describe the asymptotic behaviour of our flow in the noncollapsing case that the length of the shortest closed geodesic in the domain remains bounded below by some uniform positive constant.
In Section \ref{SYpf} we show that when the initial map is incompressible, then we can be sure that the flow will collapse neither at finite nor infinite time.
In Appendix \ref{reformulation} we explain how the flow can be reformulated essentially as a flow of pairs $(u,\Gamma)$ where $\Gamma$ is a group of isometries of the universal cover of the domain, and $u$ is a map from that universal cover which is invariant under the action of $\Gamma$. This viewpoint allows us in particular to connect our work with earlier work of Ding-Li-Liu \cite{DLL} in the genus 1 case.

Since this paper first appeared in 2012, there have been a number of developments concerning the flow we introduce here. 
In \cite{RTZ}, the authors together with Miaomiao Zhu gave a description of the asymptotics of the flow in the case that 
the domain surface degenerates at infinite time,
and this was refined together with Tobias Huxol in \cite{HRT}. 
The simplest way of developing this theory is via a new Poincar\'e inequality for quadratic differentials that is proved 
in \cite{RT-2}, which extends Lemma 2.1 to arbitrarily degenerate surfaces. 
The asymptotic analysis in \cite{RTZ} assumes that the domain surface does not degenerate in finite time, and this hypothesis is established in 
\cite{THMF_negcurv} in the case that the target $(N,G)$ has nonpositive sectional curvature.

\emph{Acknowledgements:} Thanks to Mario Micallef for his insightful remarks, Sebastian Helmensdorfer for discussions about \cite{DLL}, and Miaomiao Zhu and Saul Schleimer for comments. Both authors were supported by The Leverhulme Trust.

\section{An elliptic-Poincar\'e inequality for quadratic differentials}
\label{poincare}

In this section, we will prove that quadratic differentials which are almost holomorphic in the sense that the $L^1$ norm of their anti-holomorphic derivative is small, are small themselves in $L^1$, modulo any holomorphic part. In this paper we will apply the result to the Hopf differential of the flow map, at appropriate large times.

\begin{lemma}
\label{Ptypelemma}
Given an integer $\gamma \geq 1$ and real number $l>0$, there exists $C<\infty$ such that the following holds.
Suppose $M$ is the smooth closed orientable surface of genus $\gamma $, and $g$ is a metric of constant curvature $-1$ (for $\gamma \geq 2$) or $0$ (for $\gamma =1$) having unit area in the case that $\gamma =1$.  
Suppose further that the lengths of all simple closed geodesics on $(M,g)$ are bounded below by $l$. Then for any 
$C^1$ quadratic differential $\Psi=\psi dz^2$ on $M$, $z$ a complex variable compatible with $g$, we have
$$\|\psi dz^2-P_g(\psi dz^2)\|_{L^1}\leq C\|\dbar (\psi dz^2)\|_{L^1}.$$
\end{lemma}
In the above lemma, we have $\dbar (\psi dz^2)=\psi_\zbar d\zbar\tensor dz^2$, where if $z=x+iy$ then 
$\psi_\zbar:=\half\left(\pl{\psi}{x}+i\pl{\psi}{y}\right)$.

We will only prove this lemma in the case $\gamma \geq 2$ that concerns us, but the same ideas yield a proof for $\gamma =1$, and such an estimate appears to be necessary in the proof of 
Theorem 4.1 from \cite{DLL}.

\begin{rmk}
With a little more work, one could take the $L^q$ norm on the left-hand side, for any $q\in [1,2)$, but we will not need this improvement. In contrast with many other situations, we will show in future work that the constant $C$ in Lemma \ref{Ptypelemma} need not depend on $l$.
\end{rmk}

We will prove the Lemma \ref{Ptypelemma} using the following compactness result.

\begin{lemma}
\label{localrellichtypelemma}
Suppose $\psi_i:D\to\C$ is a sequence of $C^1$ functions on the unit disc $D$ in $\C$, such that
$$\|\psi_i \|_{L^1}+\|(\psi_i)_\zbar\|_{L^1}\leq C.$$
Then there exists $\psi_\infty\in L^1(D_\half,\C)$ such that after passing to a subsequence, we have
$$\psi_i \to \psi_\infty$$
in $L^1(D_\half,\C)$. 
Moreover, if $\|(\psi_i)_\zbar\|_{L^1}\to 0$, then 
$\psi_\infty$ is holomorphic.
\end{lemma}

\begin{proof} (Lemma \ref{localrellichtypelemma}.)
Whenever $0<r<\ep$ and $\psi\in C^1(D_\ep,\C)$, Cauchy's formula tells us
$$\psi(0)=\frac{1}{2\pi i}\int_{\partial D_r}\frac{\psi}{z}dz
+ \frac{1}{2\pi i}\int_{D_r}\frac{\psi_\zbar}{z}
dz\wedge d\zbar,$$
and therefore
$$\left|\psi(0)-\avint_{\partial D_r}\psi \right|
\leq C\int_{D_\ep}\frac{|\psi_\zbar|}{|z|}.$$
Denoting the $\ep$-mollification of $\psi$ by $\psi^\ep$, computed with respect to a smoothing function which is radially symmetric and supported sufficiently locally, we then have
$$\left|\psi(0)-\psi^\ep(0) \right|
\leq C\int_{D_\ep}\frac{|\psi_\zbar|}{|z|},$$
and more generally, for $\psi\in C^1(D,\C)$, $w\in D_\half$
and $\ep\in (0,\half]$,
$$\left|\psi(w)-\psi^\ep(w) \right|
\leq C\int_{D_\ep(w)}\frac{|\psi_\zbar|(z)}{|z-w|},$$
which integrates to give
\begin{equation}
\label{moll_est}
\|\psi-\psi^\ep\|_{L^1(D_\half)}\leq
C\ep \|\psi_\zbar\|_{L^1(D)}.
\end{equation}
Such a uniform estimate is strong enough to yield the desired compactness, because for each (fixed) $\eps\in(0,\frac12)$ the $\ep$-mollification of the unit ball in $L^1(D_1)$ is bounded in $W^{1,1}(D_{1/2})$ and thus precompact in $L^1(D_{1/2})$. 

Finally, suppose we have $\psi_i\in C^1(D,\C)$ such that 
$\psi_i\to\psi_\infty\in L^1(D_\half,\C)$ and 
$\|(\psi_i)_\zbar\|_{L^1}\to 0$ as $i\to\infty$.
Applying \eqref{moll_est} to $\psi_i$, and noting that
$(\psi_i)^\ep\to(\psi_\infty)^\ep$ in $L^1(D_\half)$ as $i\to\infty$, 
we find that the limiting $\psi_\infty$ agrees with its mollification:
$$\psi_\infty=(\psi_\infty)^\ep,$$
and is thus smooth. It must then also be holomorphic because
for all $\varphi\in C_c^\infty(D_\half)$
$$\int (\psi_\infty)_\zbar \varphi=-\int\psi_\infty\varphi_\zbar
=\lim_{i\to\infty}\left[-\int\psi_i\varphi_\zbar\right]
=\lim_{i\to\infty}\int(\psi_i)_\zbar\varphi =0.$$
\end{proof}

\begin{proof} (Lemma \ref{Ptypelemma}.)
First note that without loss of generality, we may make the additional assumption that $P_g(\psi dz^2)=0$.
Suppose that for some $\gamma $ and $l$, the claimed estimate is false however large we take $C$. Then on the smooth closed surface $M$ of genus $\gamma $, there exist  a sequence $g_i$ of hyperbolic metrics with no geodesic loop shorter than $l$, and a sequence of $C^1$ quadratic differentials $\Psi_i=\psi_i dz_i^2$ with $P_{g_i}(\Psi_i)=0$ 
such that $\|\Psi_i\|_{L^1}=1$, but 
$\|\dbar (\Psi_i)\|_{L^1}\to 0$. 
We may view the surfaces $(M,g_i)$ conformally as
$\Hyp \slash T_i$ for some discrete, fixed-point-free subgroups of $\PSL(2,\R)$ acting as isometries on the upper half-plane $\Hyp$ with the hyperbolic metric, 
and by Mumford's compactness theorem \cite{mumford}
we may pass to a subsequence so that $T_i\to T$ with $\Hyp / T$ another closed genus $\gamma $ hyperbolic surface. 
Here the convergence is to be understood as follows: For any open set $U$
 in $PSL(2,\R)$ containing an element of $T$ the intersection $T_i\cap U$ will be non-empty for $i$ large enough, while any compact subset of $PSL(2,\R)$ disjoint from $T$ will also be disjoint from $T_i$ for $i$ sufficiently large.
In particular, for fixed $z_0\in \Hyp$, the metric fundamental polygons
$$\Sigma(z_0,T_i):=\{z\in\Hyp\ |\ d(z,z_0)\leq d(z,\sigma z_0)\text{ for all }\sigma\in T_i\}$$
converge to $\Sigma(z_0,T)$ in the Hausdorff distance.

Lifting the quadratic differentials $\Psi_i=\psi_i dz_i^2$ 
to the whole of $\Hyp$ and applying Lemma \ref{localrellichtypelemma} over arbitrary  discs $D\subset\subset\Hyp$, we see that we may pass to a convergent subsequence and find a holomorphic quadratic differential $\Psi_\infty=\psi_\infty dz^2$ on $\Hyp / T$ which is locally the $L^1$-limit of $\Psi_i$ after lifting to $\Hyp$, and in particular with 
$\|\Psi_\infty \|_{L^1}=1$. 
In fact, any sequence $\Th_i$ of \emph{holomorphic} quadratic differentials on $\Hyp \slash T_i$  with unit $L^2$ length 
must subconverge to a holomorphic quadratic differential on the limit $\Hyp / T$  (say, smoothly locally after lifting to $\Hyp$) also with unit $L^2$ length, 
and so a sequence of orthonormal bases for the vector spaces of 
holomorphic quadratic differentials on $\Hyp \slash T_i$ will 
converge to an orthonormal basis of such differentials on
$\Hyp / T$. In particular,
given an arbitrary holomorphic quadratic differential $\Th$ on 
$\Hyp / T$, we can find a sequence $\Th_i$ of
holomorphic quadratic differentials on $\Hyp / T_i$ 
such that $\Th_i\to\Th$ smoothly locally after lifting to $\Hyp$. Thus we find that $P_g(\Psi_\infty)=0$ 
because
$$\int_{\Hyp / T}\langle \Psi_\infty,\Th\rangle
=\lim_{i\to\infty}
\int_{\Hyp / T_i}\langle \Psi_i,\Th_i\rangle=0.$$
But $\Psi_\infty$ is holomorphic
and hence $\Psi_\infty=0$, which 
contradicts the fact that the $L^1$ norm of $\Psi_\infty$ is one.
\end{proof}

Of course, although we will not need it, the same proof ideas allow one to prove variants on surfaces of nonconstant curvature such as:
\begin{lemma}
\label{Ptypelemma_alt1}
Suppose $(M,g)$ is any closed orientable surface (which we can also then view as a Riemann surface). Suppose $\psi dz^2$ is a $C^1$ quadratic differential on $M$. Then there exists $C<\infty$ depending only on $(M,g)$ such that
$$\|\psi dz^2-P_g(\psi dz^2)\|_{L^1}\leq C\|\dbar (\psi dz^2)\|_{L^1}.$$
\end{lemma}
An alternative proof of these results can be given using normalised isothermal coordinates, compare with the proof of Lemma \ref{lemma:conv-infty} below.

\section{Asymptotic analysis in the noncollapsing case}
\label{asymp_anal}

In this section we describe the bubbling properties of the flow at infinite time in the case that the domain metric $g$ does not degenerate. We restrict our attention to the case that the genus of the domain is at least 2; for genus 1 domains, see \cite{DLL}.

Roughly speaking we prove that as $t\to \infty$ the maps $u(t)$ subconverge (up to bubbling and modification by 
diffeomorphisms) to a branched minimal immersion $\bar u$ with the same action as $u(0)$ at the level of fundamental groups. More precisely, we show

\begin{prop}
\label{asymp_prop}
Let $(u,g)$ be any global (weak) solution of \eqref{flow} which does not degenerate at infinite time, i.e. such that 
\beq
\label{ass:non-deg}
\inf\{\ell(g(t)):\, t\in [0,\infty)\}>0 \eeq
and which satisfies property (i) of Theorem \ref{existence}
with $T=\infty$. 
Then there exist a sequence $t_j\to \infty$ and a family of orientation-preserving diffeomorphisms $f_j:M\to M$ such that the maps $u(t_j)\circ f_j$ converge weakly in $H^1(M)$ to either a branched minimal immersion or a constant map $\bar u$, with the same action on $\pi_1$, that is, $(\bar u)_*\sim (u(t))_*$ for all $t\in [0,\infty)$ for which $u(t)$ is smooth.

Furthermore, there exists a set $S$ consisting of at most finitely many points such that the above convergence is continuous (and weakly in $H^2$, strongly in $W^{1,p}, p<\infty$) 
locally on $M\setminus S$.
\end{prop}

\begin{proof}
Let $(u,g)$ be as above. As discussed in Section \ref{intro}, 
because of \eqref{energy_decay_integrated} we can select a sequence
$t_j\to \infty$ such that $\tau_{g(t_j)}(u(t_j))\to 0$ and $P_{g(t_j)}(\Phi(u(t_j),g(t_j)))\to 0$ both in $L^2$, where $\Phi$ is representing the Hopf differential.

By assumption \eqref{ass:non-deg} and the Mumford compactness theorem in a slightly different form (see, for example, \cite[Appendix C]{tromba}, or \cite[Section 7]{RFnotes} 
for a discussion of more general results)
we obtain that a subsequence of the metrics $(g(t_j))$ converges in moduli space; i.e. there exists a family of orientation-preserving diffeomorphisms $f_j:M\to M$ and a limiting metric $\bar g\in \M$ such that 
$$g_j:=f_j^* [g(t_j)]\to \bar g \quad\text{ smoothly}.$$

Modifying the maps $u(t_j)$ with the same family of diffeomorphisms, $u_j:=u(t_j)\circ f_j$ we obtain a sequence $(u_j, g_j)$ of pairs 
with the following properties
\begin{enumerate}
\item[(i)] $E(u_j,g_j)\leq E_0<\infty$
\item[(ii)] $g_j\to \bar g \text{ smoothly }$
\item[(iii)] $\norm{\tau_{g_j}(u_j)}_{L^2}\to 0$
\item[(iv)] $\norm{P_{g_j}(\Phi(u_j,g_j))}_{L^2}\to 0$.
\end{enumerate}

Based on these properties we can extract a converging subsequence of $(u_j)$ using: 
\begin{lemma} \label{lemma:conv-infty}
For any sequence $(u_j,g_j)\in W^{2,2}(M,N)\times \M$ satisfying properties (i)-(iv) above, for some $\bar g\in \M$, there exist a weakly conformal harmonic map 
$\bar u:(M,\bar g)\to N$ and a finite set of points $S\subset M$ such that after passing to a subsequence we have 
$$u_j\weakto \bar u \text{ weakly in } H^1(M)$$
as well as weakly in $W^{2,2}_{loc}(M\setminus S)$ and strongly in $W^{1,p}_{loc}(M\setminus S)$ for all $p\in[1,\infty)$. 
\end{lemma}
Because of the convergence of the metrics $g_j\to \bar g$ we prove this result using arguments familiar from the theory of the harmonic map flow (cf. \cite{struweCMH}). 
\begin{proof}[Proof of Lemma \ref{lemma:conv-infty}]
A short and well-known calculation for the Hopf differential $\Phi(u,g)=\phi dz^2$ of a smooth map $u:M\to N\embed\R^K$ tells us that 
$$\phi_\zbar=2 \langle \rho^2\tau_g(u),u_z \rangle,$$
where $\rho$ is determined by $g=\rho^2|dz|^2$. 
In particular, 
\begin{equation}
\label{hopfest1}
\|\dbar (\Phi(u,g))\|_{L^1(M)}\leq \sqrt{2} \|\tau_g(u)\|_{L^2(M)}\cdot E(u,g)^\half.
\end{equation}
In our situation this implies that $\|\dbar (\Phi(u_j,g_j))\|_{L^1(M)}\to 0$. Applying Lemma \ref{Ptypelemma}, bearing in mind (iv) above, we find that 
\begin{equation}
\label{L1hopf}
\|\Phi(u_j,g_j)\|_{L^1}\to 0\qquad\text{ as }j\to\infty.
\end{equation}

The convergence of the metrics $g_j$ implies that all occuring energies are equivalent in the sense that there exists a constant $C<\infty$ such that 
$\frac{1}{C}E(u,\bar g)\leq E(u,g_j)\leq C E(u,\bar g)$ 
for every $u\in C^1(M,N)$. The sequence $(u_j)$ is thus bounded in $H^1(M)$ (with respect to $\bar g$ say) and has a weakly converging subsequence $u_j\weakto \bar u\in H^1(M)$.

The key point is that away from points where a certain amount of energy concentrates, we have better control on $u_j$ than just a $H^1$ bound. The basic regularity estimate is the following:
\begin{lemma}[e.g. \cite{Ding-Tian}]
\label{2ndderivs}
Given any closed smooth Riemannian manifold $(N,G)$ there exist numbers $\eps_0>0$ and $C<\infty$ with the following properties: Let $D_r$ be a disc in the Euclidean 
plane of radius $r>0$,
and $u\in C^2(D_{2r},N)$ with $E(u;D_{2r})\leq \eps_0$. Then 
$$\int_{D_r} \abs{\na^2 u}^2 \leq \frac{C}{r^2} E(u;D_{2r})+C\norm{\tau(u)}_{L^2(D_{2r})}^2.$$
\end{lemma}
Once we have this $\ep_0$, we can pass to a subsequence and extract a finite set of concentration points
$$S:=\{x\in M:\, \text{ for any neighbourhood }\Om \text{ of }
x, \,\limsup_{j\to \infty} E(u_j,g_j;\Om)\geq \eps_0 \}.$$
We will now derive $W^{2,2}$ estimates on the complement of $S$. For each $x\in M\backslash S$, pick $\bar r\in (0,1]$ sufficiently small so that $3\bar r$ is less than the injectivity radius at $x$ with respect to $\bar g$, and $E(u_j,g_j;B_{\bar g}(x,3\bar r))<\ep_0$ for sufficiently large $j$. 
By exploiting that $g_j\to\bar g$, we find that $2\bar r$ is less than the injectivity radius at $x$ with respect to $g_j$, and 
$E(u_j,g_j;B_{g_j}(x,2\bar r))<\ep_0$, for sufficiently large $j$. 
We now pick $j$-dependent \emph{normalised isothermal coordinates} around $x$ in the following way. Let $\vph_j:(D,g_{H})\to (M,g_j)$ be a sequence of locally isometric coverings from the Poincar\'e disc, mapping the origin to $x$ and
the first coordinate vector at the orgin to a positive multiple of a fixed vector in $T_xM$. Then the maps $\varphi_j$ converge smoothly locally, and for some $r>0$ depending only on $\bar r$, we have $\vph_j(D_{2r})\subset B_{g_j}(x,2\bar r)$.
In particular, writing $u_j$ also for the map $u_j\circ \varphi_j$ representing it in these coordinate charts, we conclude that
$$E(u_j,D_{2r})<\ep_0$$
for sufficiently large $j$, and even working with respect to the flat metric on $D_{2r}$, we have $\|\tau(u_j)\|_{L^2(D_{2r})}\to 0$
(because the Poincar\'e and Euclidean metrics are conformally equivalent on $D_{2r}$).
Applying Lemma \ref{2ndderivs} for each $j$, we obtain uniform bounds for $\|\nabla^2 u_j\|_{L^2(D_r)}$, and thus $u_j$ is bounded uniformly in $W^{2,2}(D_r)$. Passing to a subsequence, we deduce weak $W^{2,2}$ convergence to some limiting map from $D_r$ to $N$, and by \eqref{L1hopf}, that map must be weakly conformal.
Repeating the argument elsewhere in $M\backslash S$, and 
exploiting once more that $g_j\to \bar g$, 
we obtain that $u_j$ converges weakly in $W^{2,2}_{loc}(M\backslash S,N)$ to a limit, which must be weakly conformal with respect to $\bar g$. But that limit must then coincide with the weak-$H^1$ limit $\bar u$.
By passing to the limit in the equation satisfied by $u_j$ (i.e. in the definition of $\tau_{g_j}$) we establish that $\bar u$ must be harmonic in $M\backslash S$, and therefore smooth and harmonic throughout $M$ by the removability of point singularities \cite{SU}.
\end{proof}

Having proved Lemma \ref{lemma:conv-infty}, we continue with the proof of Proposition \ref{asymp_prop}.
Consider now the action  
$$u(t)_*:\pi_1(M,x)\to \pi_1(N,u(x,t))$$ on the fundamental groups induced by a (weak) solution $(u,g)$ of \eqref{flow} satisfying property (i) of Theorem \ref{existence}. As the set of singular points $\cup_{i}\{T_i\}\times S(T_i)$ is finite, we may choose a basepoint $x$ and a generating set of elements $\{\gamma_k\}$ of $\pi_1(M,x)$ that are disjoint from $\cup_{i}S(T_i)$. For any $t,t'\in [0,T)$, the path $\alpha_{t,t'}(s):=u(x,t+s(t'-t)), s\in [0,1]$, thus induces a change-of-basepoint isomorphism for which the loops  
$$u(t)\circ\gamma_k \text{ and } \alpha_{t,t'}^{-1} \cdot (u(t')\circ \gamma_k)\cdot \alpha_{t,t'}$$ represent the same element of $\pi_1(N,u(x,t))$. We thus conclude that the flow \eqref{flow} leaves the action on the fundamental group invariant.  

Recall now that by hypothesis, the metric $g(t)$ degenerates neither at finite nor infinite time and consider again the maps $u_i=u(t_i)\circ f_i$ we fed into Lemma \ref{lemma:conv-infty}, and the limits $\bar u$ and $\bar g$, and singular set $S$ that we obtained in return. Selecting a base point $x$ and generating curves $\{\gamma_k\}$ of $\pi_1(M,x)$ disjoint from $S$, we obtain that the curves $u_i\circ \gamma_k$ converge uniformly to the limiting curves $\bar u\circ \gamma_k$. Conjugating with an appropriate sequence of paths $\alpha_i$ (converging to a constant) we thus obtain for the induced action on the fundamental groups
\beq \label{eq:action} \bar u_*\sim (u_i)_*=(u(t_i)\circ f_i)_*=u(t_i)_*\circ (f_i)_*\sim (u_0)_*\circ (f_i)_*\eeq 
for $i$ sufficiently large, say $i\geq i_0$. 
In a slight abuse of notation, here and in the following we use $(\cdot)_*$ also to denote the action of maps on fundamental groups with basepoints different from $x$, such as $(u_0)_*$ acting on $\pi_1(M,f_i(x))$,
with the basepoint determined such that the composition of the above homomorphisms between fundamental groups is well defined. 
Furthermore we extend the equivalence relation $\sim$ to general homomorphisms between fundamental groups.

We cannot conclude from \eqref{eq:action} that $(u_0)_*\sim (\bar u)_*$ since the diffeomophisms $f_i$ are in general not homotopic to the identity. However replacing $f_i$ in the above construction by $\tilde f_i=f_i\circ f_{i_0}^{-1}$ gives a new sequence of maps 
$\tilde u_i:=u(t_i)\circ \tilde f_i
=u_i\circ f_{i_0}^{-1}$ 
that converge to the map $ \bar u\circ f_{i_0}^{-1}$, which is weakly conformal and harmonic with respect to $(f_{i_0}^{-1})^*\bar g$ and has the same action on the fundamental groups as $u_0$ because
\begin{equation}
\label{rodent}
\begin{aligned}
(\bar u \circ f_{i_0}^{-1})_* &=\bar u_* \circ (f_{i_0}^{-1})_*
\sim (u_i)_* \circ (f_{i_0}^{-1})_*
= (u(t_i)\circ f_i)_* \circ (f_{i_0}^{-1})_*\\
&=(u(t_i))_*\circ (f_i \circ f_{i_0}^{-1})_*
\sim (u_0)_* \circ (\tilde f_i)_*
\end{aligned}
\end{equation}
for all $i\geq i_0$, and specialising to $i=i_0$ (note that 
$\tilde f_{i_0}$ is the identity) we find that
\begin{equation}
\label{panda}
(\bar u \circ f_{i_0}^{-1})_*=(u_0)_*.
\end{equation}
\end{proof}
\begin{rmk}
We do not claim that the diffeomorphisms $(\tilde f_i)_{i\in \N}$ are homotopic to the identity in general; we do not even claim that they cannot diverge within the mapping class group.
\end{rmk}

\begin{rmk}
When $\pi_2(N)=0$, then the branched minimal immersion $\bar u \circ f_{i_0}^{-1}$ obtained in the above proof is homotopic
to $u_0$, after possibly increasing $i_0$. Indeed, for sequences of maps $v_i:M\to N$ that converge (uniformly) away from a finite number of points the assumption that $\pi_2(N)=0$ implies that the maps $v_i$ are all in the same homotopy class as the limit map provided $i$ is sufficiently large. This means that the flow does not change the homotopy class of $u(t)$ at finite time, nor at infinity, provided we use the diffeomorphisms $\tilde f_j=f_j\circ f_{i_0}^{-1}$ where
$i_0$ is large enough so that
$u_{i_0}$ is homotopic to $\bar u$. 
\end{rmk}
 
\subsection{A proof of Theorem \ref{minsurfexist} of Schoen-Yau and Sacks-Uhlenbeck}
\label{SYpf}
The cases that $M$ has genus 0 or 1 are handled by the classical harmonic map flow and \cite{DLL} respectively. Therefore we will assume that the genus of $M$ is at least 2.
Let $u_0$ be an incompressible map and let $g_0$ be a hyperbolic metric. By the results of the previous section it is enough to show that the corresponding solution of \eqref{flow} exists for all time and that the domain metrics $g(t)$ do not degenerate at infinity.

In fact, since the maps $u(t)$ are incompressible for all $t$, any degeneration of the metric would cause the energy to become unbounded according to 

\begin{lemma}[cf. \cite{tromba}] %p77
\label{lemma:incomp}
Let $(M,g)$ be a closed orientable hyperbolic surface of genus $\gamma\geq 2$
%Let $M$ be a closed surface of genus $\geq 2$ 
and let $(N,G)$ be a closed Riemannian manifold. Then the energy of an incompressible map 
$u:M\to N$ is bounded from below by
$$E(u,g)\geq \frac{\varphi(\ell)}{\ell}$$
where $\ell$ is the length of the shortest closed geodesic on $(M,g)$ and $\varphi(\ell):=c^2 (\pi/2-\arctan(\sinh(\frac{\ell}{2})))\to \frac{c^2\pi}{2}$ as $\ell\to 0$, where $c>0$ is the length of the shortest closed geodesic in the target $(N,G)$.
\end{lemma}

The proof is based on the collar lemma which gives the following explicit description of hyperbolic surfaces near closed 
geodesics:

\begin{lemma}[Keen-Randol \cite{randol}] \label{lemma:collar}
Let $(M,g)$ be a closed orientable hyperbolic surface and let $\gamma$ be a simple closed geodesic of length $\ell$. Then there is a neighbourhood $U$ around $\gamma$, a so called collar, which is isometric to the 
cylinder 
$$P=(-X(\ell),X(\ell))\times S^1$$
equipped with the metric $\rho^2(x)(dx^2+d\theta^2)$ where 
$$\rho(x)=\frac{\ell}{2\pi \cos(\frac{\ell x}{2\pi})} 
\qquad\text{ and }\qquad  
X(\ell)=\frac{2\pi}{\ell}\left(\frac\pi2-\arctan\left(\sinh\left(\frac{\ell}{2}\right)\right) \right).$$ 
The geodesic $\gamma$ then corresponds to the circle $\{(0,\theta)\ |\ \theta\in S^1\}\subset P$.
\end{lemma}

In this version of the collar lemma, the intrinsic distance $w$ between the two ends of the collar is related to $\ell$ via
$$\sinh \frac{\ell}{2} \sinh \frac{w}{2}=1,$$
which is sharp.

\begin{proof}[Proof of Lemma \ref{lemma:incomp}] 
Let $g\in \M$ and let $U$ be the collar neighbourhood around the shortest closed geodesic of $(M,g)$. Given any map $u:M\to N$ we consider the energy of the corresponding map $\tilde u:P\to N$ on the cylindrical collar. If $u$ is incompressible, the curves $\beta_x:\theta\mapsto \tilde u(x,\theta)$ are closed and topologically nontrivial and thus have length $L(\beta_x)\geq c>0$.
Using the conformal invariance of the Dirichlet energy we can thus estimate
\begin{equation*}
E(u,g)\geq E(\tilde u,P)\geq\int_{-X}^X E(\beta_x) dx\geq \frac1{4\pi}\int_{-X}^X L(\beta_x)^2 dx\geq \frac{c^2}{2\pi}\cdot X(\ell)
\end{equation*}
which implies the claim.
\end{proof}

Given any solution $(u,g)$ of \eqref{flow} with an incompressible initial map $u_0$ we thus obtain that the length $\ell(g(t))$ of the shortest closed geodesic of
$(M,g(t))$ is bounded away from zero by a constant $c=c(E(u_0,g_0), (N,G))>0$. Consequently the solution constructed in Theorem \ref{existence} exists for all time and as $t\to \infty$ the maps
$u(t)$ subconverge (after modification with diffeomorphisms) to a branched minimal immersion $\bar u$ with the same action as $u_0$ on $\pi_1(M)$.

\appendix

\section{A reformulation of the flow equations}
\label{reformulation}

Up to now, we considered the flow as an evolving pair $(u,g)$ of map $u:M\to N$ and metric $g\in\m_c$.
In this section we change viewpoint by lifting the whole picture to the universal cover $(\tilde M,\tilde g)$ of $(M,g)$, which is the Euclidean plane if the genus $\gamma=1$ or the hyperbolic plane $\Hyp$ if $\gamma\geq 2$, and see the flow essentially as an evolving pair $(\tilde u,\Gamma)$, where $\Gamma$ is
a group of isometries of $(\tilde M,\tilde g)$ and
$\tilde u:\tilde M\to N$ is a map which is invariant under the action of $\Gamma$.

In the genus one case, this will show that our flow corresponds to the flow of Ding-Li-Liu \cite{DLL}.
In the higher genus case, it will give a different viewpoint which might be used as a basis for an alternative short-time existence proof, for example.
In general, this approach has the advantage that it avoids the quotient by the \emph{infinite-dimensional} group of 
diffeomorphisms used in the definition of $\cal A$.

Strictly speaking, it is natural to quotient the manifold of pairs
$(\tilde u,\Gamma)$ by the natural action of the group of isometries of $(\tilde M, \tilde g)$ -- i.e. for any such isometry $\si$, we identify 
$(u\circ \si, \si^{-1}\Gamma \si)$ with $(\tilde u, \Gamma)$ --
but we only emphasise this viewpoint in the case of genus $\gamma\geq 2$.

\subsection{The genus 1 case}
\label{torus_sect}

For $(\al,\be)\in\R^2$, with $\be\neq 0$, consider the group $\Gamma_{\al,\be}$ of translations of the Euclidean plane which fix the lattice generated by $(1/\be,0)$ and $(\al,\be)$. Then the quotient $M_{\al,\be}:=(\R^2,g_E) / \Gamma_{\al,\be}$ of the Euclidean plane is a flat torus of unit area. (Here we are writing $g_E:=dx^2+dy^2$, and will use the same notation for the resulting metric on the quotient.)

We will be switching between viewing the domain of the map part of the flow as $(M_{\al,\be},g_E)$ and as a fixed underlying manifold, say $M:=M_{0,1}$, the square torus, but with a varying metric as in the introduction.
We pass between the two viewpoints via the linear map
$L_{\al,\be}:\R^2\to\R^2$ which sends $(1,0)$ to $(1/\be,0)$ and $(0,1)$ to $(\al,\be)$; given as a matrix, this is
$$\left(
\begin{array}{cc} 1/\be & \al \\ 0 & \be
\end{array}
\right).$$
The flat metrics we consider on $M$ will then be the 2-parameter family $L_{\al,\be}^*(g_E)$. We will be able to see the flow as moving the domain metric on $M$ within this two-dimensional submanifold of the infinite dimensional manifold of unit area flat metrics on $M$ because:

{\bf Claim:}
The two-dimensional submanifold of metrics $\{L_{\al,\be}^*(g_E)\}$ on $M$ is \emph{horizontal}. Explicitly, if we take a smooth one-parameter family 
$(\al+s \dot\al,\be+s\dot\be)$
for $s$ in a neighbourhood of $0$,  
then the metrics on $M_{\al,\be}$ defined by 
\begin{equation}
g(s):=(L_{\al,\be}^{-1})^*(L_{\al+s\dot\al,\be+s\dot\be})^*(g_E)
=\left(L_{\al+s\dot\al,\be+s\dot\be} \circ (L_{\al,\be}^{-1})\right)^*(g_E)
\end{equation}
satisfy
$$\pl{g}{s}\bigg|_{s=0}=Re\left[
\left(
-\frac{2\dot\be}{\be}-i(\frac{\al\dot\be}{\be^2}+\frac{\dot\al}{\be})
\right)dz^2
\right]$$
and the right-hand side is the real part of a holomorphic quadratic differential.

\begin{proof}
By definition of the Lie derivative, we have 
$$\pl{g}{s}\bigg|_{s=0}=\lie_X(g_E),$$
where
$$X=\frac{d}{ds}\bigg|_{s=0}L_{\al+s\dot\al,\be+s\dot\be} \circ (L_{\al,\be}^{-1}).$$
(Note $X$ is a vector field on $\R^2$, not on a torus.)
A short calculation then yields
$$X(x,y)=\left(-\frac{\dot\be}{\be}x
+\left(\frac{\al\dot\be}{\be^2}+\frac{\dot\al}{\be}\right)y,
\frac{\dot\be}{\be}y\right).$$
\cmt{Here is the calculation, to aid checking: We equate linear maps with their matrices!
$$L_{\al+s\dot\al,\be+s\dot\be}=\left(
\begin{array}{cc} 1/(\be+s\dot\be) & \al+s\dot\al \\ 0 & \be+s\dot\be
\end{array}
\right)
\qquad
L_{\al,\be}^{-1}=\left(
\begin{array}{cc} \be & -\al \\ 0 & 1/\be
\end{array}
\right)
,$$
so the product matrix (corresponding to) $L_{\al+s\dot\al,\be+s\dot\be} \circ (L_{\al,\be}^{-1})$ is
$$\left(
\begin{array}{cc} \frac{\be}{\be+s\dot\be} & -\frac{\al}{\be+s\dot\be}+\frac{\al+s\dot\al}{\be} \\ 0 & \frac{\be+s\dot\be}{\be}
\end{array}
\right).$$
Differentiate wrt $s$ and set $s=0$:
$$\frac{d}{ds}\bigg|_{s=0}L_{\al+s\dot\al,\be+s\dot\be} \circ (L_{\al,\be}^{-1})=
\left(
\begin{array}{cc} -\frac{\dot\be}{\be} & \frac{\al\dot\be}{\be^2}+\frac{\dot\al}{\be} \\ 0 & \frac{\dot\be}{\be}
\end{array}
\right)...$$
}

Viewing $\lie_X g$ as the symmetric gradient of $X$ 
we then have
\begin{equation}
\label{horizontal_formula}
\begin{aligned}
\pl{g}{s}\bigg|_{s=0}&=
2\left[
-\frac{\dot\be}{\be}(dx^2-dy^2)
+\left(\frac{\al\dot\be}{\be^2}+\frac{\dot\al}{\be}\right)dx\,dy
\right]\\
&= Re\left[
\left(
-\frac{2\dot\be}{\be}-i(\frac{\al\dot\be}{\be^2}+\frac{\dot\al}{\be})
\right)dz^2
\right].
\end{aligned}
\end{equation}
\end{proof}

As remarked above, the fact that the metrics $\{L_{\al,\be}^*(g_E)\}$ form a horizontal submanifold allows us to see a flow solution $(u,g)$ to \eqref{flow}, in the genus one case, as a time-dependent map $u:M\to N$ together with a domain metric $\{L_{\al(t),\be(t)}^*(g_E)\}$ on $M$. We now compute the induced evolution of $\al(t)$ and $\be(t)$ under this flow. We will use these formulae only to make the connection with \cite{DLL}.

It is convenient to write $u=U\circ L_{\al,\be}$, so that we can view our flow map $u$ as a map $U$ defined on $M_{\al,\be}$.
In that case, the Hopf differential is defined to be
$$\phi dz^2=(|U_x|^2-|U_y|^2-2i\langle U_x,U_y\rangle)dz^2,$$
and because the holomorphic quadratic differentials are here precisely the parallel quadratic differentials (i.e.~constant multiples of $dz^2$ using this global coordinate) we have
\begin{equation}
\label{squirrel}
\pl{g}{t}=\frac{\eta^2}{4}Re[P_g(\phi dz^2)]=\frac{\eta^2}{4}Re\left[\Big(\int_{M_{\al,\be}}|U_x|^2-|U_y|^2-2i\langle U_x,U_y\rangle\Big)\cdot dz^2\right].
\end{equation}
We can now compare \eqref{squirrel} with \eqref{horizontal_formula} to show that the flow makes $\al$ and $\be$ evolve according to
$$\dot\be=-\frac{\eta^2}{4}\frac{\be}{2}\int_{M_{\al,\be}} |U_x|^2-|U_y|^2$$
and
$$\dot\al=\frac{\eta^2}{4}\int_{M_{\al,\be}} \frac{\al}{2}\left(|U_x|^2-|U_y|^2\right)
+2\be\langle U_x,U_y\rangle.$$
Thus, as desired, we can see the flow as evolving the subgroup 
$\Gamma_{\al,\be}$ together with a $\Gamma_{\al,\be}$ invariant map $\tilde u:\R^2\to N$ arising by lifting $U$.
The map $\tilde u$ will then evolve according to
$$\pl{\tilde u}{t}=\tau(\tilde u)-d\tilde u(X)$$
where the tension $\tau$ is computed with respect to the Euclidean
metric.

Finally, we would like to compare the flow with that of Ding-Li-Liu \cite{DLL}, which first requires us to rewrite the formulae in terms of $u$ rather than $U$.
By definition, we have $U(x,y)=u(\be x-\al y,y/\be)$, giving
$U_x=\be u_x$ and $U_y=-\al u_x+\frac{u_y}{\be}$, and hence
(after a little manipulation)
$$\dot\be=-\frac{\eta^2}{4}\frac{\be}{2}\int_M (\be^2-\al^2)|u_x|^2-
\frac{|u_y|^2}{\be^2}+\frac{2\al}{\be}\langle u_x,u_y\rangle$$
and
$$\dot\al=\frac{\eta^2}{4}\int_M |u_x|^2\left(-\frac{3}{2}\al\be^2-\frac{\al^3}{2}\right)-\frac{\al}{2\be^2}|u_y|^2
+\left(\frac{\al^2}{\be}+2\be\right)\langle u_x,u_y\rangle.$$
Moreover, in \cite{DLL} the conformal tori $\Si_{a,b}$ arising by taking the quotient of $\R^2$ by the lattice generated by $(1,0)$ and $(a,b)$ were considered.
Because $\Si_{\al\be,\be^2}$ arises as a dilation of $M_{\al,\be}$ by a factor $\be$, we thus find that $a=\al\be$ and $b=\be^2$, and therefore (after a little more manipulation)
$$\dot a=\frac{\eta^2}{4}2b\int_M \langle u_x,u_y\rangle-a|u_x|^2,$$
and
$$\dot b=\frac{\eta^2}{4}\int_M -(b^2-a^2)|u_x|^2+|u_y|^2-2a\langle u_x,u_y\rangle$$
which agrees with the formulae in \cite{DLL} in the case that $\eta^2=2$.

\subsection{The higher genus case}

There is an analogue of the construction in the previous section which would provide an alternative approach to some of the theory of this flow. However, since we don't use it directly in this paper, we just sketch the construction.
Despite the genus one construction, we cannot find a $6\ga -6$
dimensional horizontal slice within the space of hyperbolic metrics, and this prevents us from mimicking the construction of the maps $L_{\al,\be}$, etc.

In this case, the flow is seen to move within the space of pairs
$(\tilde u, \Gamma)$, where $\Gamma$ is a group of isometries of   $\Hyp$ such that $\Hyp / \Gamma$ is a closed orientable surface of fixed genus and $\tilde u$ is a map from $\Hyp$ which is invariant under the action of $\Gamma$ (i.e. for $\si\in \Gamma$, we have $\tilde u=\tilde u\circ \si$).
More precisely, the flow moves within the quotient of this space under the action of $\PSL(2,\R)$ -- i.e. if $\si\in \PSL(2,\R)$ then we identify $(\tilde u, \Gamma)$ with $(\tilde u \circ \si,\si^{-1}\Gamma \si)$. 

Now there is a correspondence between holomorphic quadratic differentials $\psi dz^2$ on $\Hyp / \Gamma$ (whose real parts represent tangent vectors in the space of hyperbolic metrics) lifted to $\Hyp$, and elements of
$$V:=\{X\in\Gamma(T\Hyp)\ |\ \lie_X g_\Hyp=Re[\psi dz^2]\}$$ 
with elements identified if they differ by a Killing vector field on $\Hyp$. Vector fields $X$ representing elements of $V$ can be viewed as tangent vectors in the space of subgroups $\Gamma$ as above (intuitively, one pushes a whole fundamental polygon 
in the direction of $X$ to give the new fundamental polygon) and we retain the notation $X$ to represent this tangent vector.
The flow may then be written
$$\pl{\Gamma}{t}=X(\Gamma,\tilde u);\qquad \pl{\tilde u}{t}=\tau(\tilde u)-d\tilde u(X(\Gamma,\tilde u))$$
where $\tau$ is the tension with respect to the hyperbolic metric.

\section{Viewing $\A$ as a manifold}
\label{banach_manifold}

In Section \ref{flow_def_sect} we defined $\A$, and pretended that it was a manifold in order to formally derive our new flow equations.
Although we do not need to make that discussion rigorous for the purposes of this paper, in this section we sketch how to adapt the definition of $\A$ in order to obtain a smooth Banach manifold, and we describe its tangent bundle.
We do that by following 
the path laid out in Teichm\"uller theory (see \cite[Chapter 2]{tromba}) %page 50
but restrict to the case $\ga>1$ for simplicity.

For $s>0$ sufficiently large (as discussed in \cite{tromba}, $s>3$ is large enough) we  relax the regularity of $u$, $g$ and $f$ by asking that $u\in H^s(M,N)$ and $g\in\M^s$ (the metrics of constant curvature as in $\M$, but of regularity $H^s$) and 
$f\in \D_0^{s+1}$, the group of diffeomorphisms $f:M\to M$ 
of regularity $H^{s+1}$ that are isotopic to the identity.
We can then define 
$$\A^s:=(H^s(M,N)\times\M^s)/\D_0^{s+1},$$
and it is this that will be a Banach manifold.

By consideration of isothermal coordinates, any element of $\A^s$ can be represented by a pair $(u_0,g_0)$ with $g_0$ \emph{smooth} and $u_0\in H^s$. By the slice theorem (\cite[Theorem 2.4.3]{tromba}) there exists a $6\ga-6$ dimensional submanifold 
$\cal S$ of $\M^s$ passing through $g_0$, 
each element of which is smooth, 
such that metrics $g\in\M^s$ near to $g_0$ can be written uniquely as $f^*g_s$
for some $g_s\in {\cal S}$ and $f\in \D_0^{s+1}$ near the identity.
(Moreover, $f$ and $g_s$ depend smoothly on $g$, and $f$ is unique within the whole of $\D_0^{s+1}$.)
Therefore, any $(u,g)$ near to $(u_0,g_0)$ represents the same element in $\A^s$ as $(u\circ f^{-1},g_s)$,
and we can locally identify a neighbourhood of $[(u_0,g_0)]$ in 
$\A^s$ with a neighbourhood of $(u_0,g_0)$ in the manifold
$H^s(M,N)\times {\cal S}$, which is enough to establish a manifold structure on $\A^s$ once one has checked that the structure is independent of the point $g_0$ where we centred our slice.
Meanwhile, by construction,
the tangent space in the slice ${\cal S}$ at $g_0$ is given by 
$${\cal H}:=\{Re(\phi dz^2)\ |\ \phi dz^2 \text{ is a holomorphic quadratic differential on }(M,g_0)\},$$
and we see that we can identify the tangent space of $\A^s$ 
at $[(u_0,g_0)]$ with
$H^s(u_0^*(TN))\times {\cal H}$.
It then makes sense to define a metric as in 
Section \ref{flow_def_sect} to be 
the $L^2$ norm  on each of the two
components of the tangent space (weighted as desired).
Further details of the classical construction of this type
can be found in \cite{tromba}.

It may be worth noting that if we had defined $\A^s$ by taking the quotient by \emph{all} $H^{s+1}$ diffeomorphisms rather than just those homotopic to the identity, then $\A^s$ would not have become a manifold.

{\sc Mathematical Institute, University of Oxford, Oxford, OX2 6GG, UK}

{\sc Mathematics Institute, University of Warwick, Coventry,
CV4 7AL, UK}

\end{document}